\documentclass[11pt,reqno]{amsart}
\usepackage{amsmath}
\usepackage{cases}
\usepackage{mathrsfs}
\usepackage{bbm}
\usepackage{amssymb}
\usepackage{amscd}
\usepackage{amsfonts,latexsym,amsmath,amsthm,amsxtra,mathdots,amssymb,latexsym,mathabx}
\usepackage[all,cmtip]{xy}
\RequirePackage{amsmath} \RequirePackage{amssymb}
\usepackage{color}
\usepackage{colordvi}
\usepackage{multicol}
\usepackage{hyperref}
\usepackage{mathtools}
\usepackage[margin=1in]{geometry}
\usepackage{xcolor}
\hypersetup{
    colorlinks,
    linkcolor={red!50!black},
    citecolor={blue!50!black},
    urlcolor={blue!80!black}
}
\usepackage{cite}

\newcommand{\bea}{\begin{eqnarray}}
\newcommand{\eea}{\end{eqnarray}}
\newcommand{\bna}{\begin{eqnarray*}}
\newcommand{\ena}{\end{eqnarray*}}

\numberwithin{equation}{section}

\theoremstyle{plain}
\newtheorem{lemma}{Lemma}[section]
\newtheorem{theorem}[lemma]{Theorem}

\theoremstyle{definition}

\newtheorem{remark}{Remark}

\renewcommand{\Re}{\operatorname{Re}}
\renewcommand{\Im}{\operatorname{Im}}

\begin{document}

\title{On some estimates involving Fourier coefficients of Maass cusp forms}

	\author{Qingfeng Sun}
	\address{School of Mathematics and Statistics, Shandong University, Weihai\\Weihai, Shandong 264209, China}
	\email{qfsun@sdu.edu.cn}

    \author{Hui Wang}
    \address{Department of Mathematics, Shandong University, Jinan 250100, China}
    \email{wh0315@mail.sdu.edu.cn}



\begin{abstract}
Let $f$ be a Hecke-Maass cusp form for $\rm SL_2(\mathbb{Z})$
with Laplace eigenvalue $\lambda_f(\Delta)=1/4+\mu^2$ and let $\lambda_f(n)$ be its $n$-th
normalized Fourier coefficient.
It is proved that, uniformly in $\alpha, \beta \in \mathbb{R}$,
$$
\sum_{n \leq X}\lambda_f(n)e\left(\alpha n^2+\beta n\right)
\ll X^{7/8+\varepsilon}\lambda_f(\Delta)^{1/2+\varepsilon},
$$
where the implied constant depends only on $\varepsilon$. We also consider
the summation function of $\lambda_f(n)$ and under the Ramanujan conjecture
we are able to prove
$$
\sum_{n \leq X}\lambda_f(n)\ll
X^{1/3+\varepsilon}\lambda_f(\Delta)^{4/9+\varepsilon}
$$
with the implied constant depending only on $\varepsilon$.
\end{abstract}

\thanks{Qingfeng Sun is partially supported by the National Natural Science Foundation
of China (Grant Nos. 11871306 and 12031008)}

\maketitle

\section{Introduction}

The Fourier coefficients of automorphic forms contain many
mysterious properties, especially its oscillation properties,
which have been intensively studied by many number theorists.
Let $\lambda_F(n)$ be the normalized
Fourier coefficients of an automorphic
form $F$ on $\rm GL_n$.
Among the many criteria for evaluating the nature of oscillation,
the summation function $\sum_{n\leq X}\lambda_F(n)$ is certainly the most basic one,
and the related exponential sum $\sum_{n\leq X}\lambda_F(n)e\left(g(n)\right)$, where
as usual, $e(x)=e^{2\pi i x}$ and $g(n)$ is a real-valued function, is another important target,
with obvious application indications. In this paper, we are concerned with the oscillating behavior of the Fourier coefficients of
cusp forms on the full modular group. More precisely,
let $f$ be a holomorphic Hecke cusp form of weight $k$ or a Hecke-Maass cusp form
of Laplace eigenvalue $\lambda_f(\Delta)=1/4+\mu^2$ $(\mu>0)$ for $\rm SL_2(\mathbb{Z})$ with
normalized Fourier coefficients $\lambda_f(n)$.
Then for $f$ holomorphic, one has the Fourier expansion
\bna
f(z)=\sum_{n\geq 1}\lambda_f(n) n^{(k-1)/2}e(nz), \quad \Im(z)>0,
\ena
while for $f$ a Maass form, we can write its Fourier expansion as
\bna
f(z)=\sqrt{y}\sum_{n\neq 0}\lambda_f(n)K_{i\mu}(2\pi |n|y)e(nx),
\ena
where $K_{i\mu}$ is the modified Bessel function of the third kind.
The famous Ramanujan--Petersson conjectures assert that
$\lambda_f(n)\ll_{\varepsilon} n^{\varepsilon}$ for any $\varepsilon>0$.
This was proved by Deligne \cite{Deligne} for $f$ holomorphic.
For $f$ a Maass cusp form, the best result is $\lambda_f(n)\ll_{\varepsilon} n^{7/64+\varepsilon}$
due to Kim and Sarnak \cite{KS}.
By Rankin--Selberg theory (see \cite[Proposition 19.6]{DFI2}), we have the following average result
\bea\label{GL2: Rankin Selberg}
\sum_{n\leq X}|\lambda_f(n)|^2\ll_{\varepsilon}X(X|\mu|)^{\varepsilon}.
\eea

The Rankin-Selberg's estimate in \eqref{GL2: Rankin Selberg} shows that the Fourier coefficients $\lambda_f(n)$
behave like constants on average. However, as $n$ grows $\lambda_f(n)$ in fact vary greatly
in sign, which can be seen in the estimate (see \cite{Epstein, Hafner})
\bna
\sum_{n\leq X}\lambda_f(n)e(n \alpha)\ll_{f} X^{1/2}\log (2X),
\ena
which holds uniformly in $\alpha\in \mathbb{R}$.
It's worth noting that the above estimate depends on the form $f$, whereas in certain applications
such the subconvexity problem, one may need an explicit dependence on the form $f$.
To this end, Godber \cite{Godber} proved that for any $\alpha\in \mathbb{R}$
and any $\varepsilon>0$,
\bea\label{linear}
\sum_{n \leq X}\lambda_f(n)e\left(n\alpha\right)
\ll X^{1/2+\varepsilon}\lambda_f(\Delta)^{1/4+\varepsilon},
\eea
where the implied constant depends only on $\varepsilon$. This is an improvement over the
bound $X^{1/2+\varepsilon}\lambda_f(\Delta)^{1/2+\varepsilon}$ in \cite[Section 8.3]{I}.
For the associated quadratic exponential sums,  Pitt \cite{Pitt} first proved
for any $\alpha, \beta\in \mathbb{R}$
and any $\varepsilon>0$,
\bea\label{P result}
\sum_{n \leq X}\lambda_f(n)e\left(\alpha n^2+\beta n\right)\ll X^{15/16+\varepsilon},
\eea
where the implied constant depends only on the form $f$ and $\varepsilon$.
Later, Liu and Ren \cite{LR} improved the above bound to $X^{7/8+\varepsilon}$.

It is interesting and useful to prove an estimate for the exponential sum in
\eqref{P result}, which does not depend on the form $f$. So the first aim of our paper is
to prove the following.

\begin{theorem}\label{the quadratic exponential sums}
Let $\lambda_f(n)$ be the normalized Fourier coefficients
of a Hecke-Maass cusp form for $\rm SL_2(\mathbb{Z})$
with Laplacian eigenvalue $\lambda_f(\Delta)=1/4+\mu^2$.
For any $\alpha, \beta \in \mathbb{R}$ and any $\varepsilon>0$, we have
\[
\sum_{n \leq X}\lambda_f(n)e\left(\alpha n^2+\beta n\right)\ll
X^{7/8+\varepsilon}\lambda_f(\Delta)^{1/2+\varepsilon},
\]
where the implied constant depends only on $\varepsilon$.
\end{theorem}

\begin{remark}
The methods of proving Theorem \ref{the quadratic exponential sums} can also be adopted for
holomorphic forms. Let $f$ be a holomorphic cusp form of weight $k$ for $\rm SL_2(\mathbb{Z})$
with normalized Fourier coefficients $\lambda_f(n)$. It would have,
uniformly in $f$ and $\alpha, \beta \in \mathbb{R}$,
\[
\sum_{n \leq X}\lambda_f(n)e\left(\alpha n^2+\beta n\right)\ll
X^{7/8+\varepsilon}k^{1+\varepsilon}.
\]
\end{remark}

The proof for Theorem \ref{the quadratic exponential sums} is based on the ideas introduced by
\cite{Pitt}. That is, taking two different approaches according to the
Dirichlet approximation of $\alpha$.
Let $Q\geq 1$ be any given number. By Dirichlet's theorem,
for any $\alpha\in\mathbb{R}$,
there exists a reduced rational number $\ell/q$ with $1\leq q\leq Q$
such that
\bea\label{Dirichlet's theorem}
\left|\alpha-\frac{\ell}{q}\right|\leq\frac{1}{qQ}.
\eea
For ``larger" $q$, i.e., the oscillation of the exponential function is large,
we separate the Fourier coefficients $\lambda_f(n)$ and the exponential function by the $\delta$-method
as in Pitt \cite{Pitt} and deal the resulting sum using techniques in Liu and Ren \cite{LR}. However, with the demand that we shall make clear the dependence on $f$ in mind,
we need to adopt a different form of the $\delta$-method, i.e., the $\delta$-method due to
Duke, Friedlander and Iwaniec (see Section 2.5). For ``smaller" $q$, we deal with it as in
Pitt \cite{Pitt}, except for applying the result of Godber in \eqref{linear}.

Another purpose of our paper is to consider the summation function of $\lambda_f(n)$, i.e.,
\bna
\sum_{n\leq X}\lambda_f(n).
\ena
For holomorphic cusp forms, this was first studied by
Hecke \cite{Hecke} in 1927.
Later, Walfisz \cite{Walfisz} proved the following estimate,
i.e.,
\bna
\sum_{n\leq X}\lambda_f(n)\ll_{f,\vartheta} X^{\frac{1+\vartheta}{3}},
\ena
where $\vartheta$ is exponent towards the Ramanujan-Petersson conjecture, i.e.,
$|\lambda_f(n)|\leq n^\vartheta$.
Then by Deligne \cite{Deligne}, one has the estimate $O_{f,\varepsilon}(X^{1/3+\varepsilon})$
for any $\varepsilon>0$.
Subsequently, Hafner and Ivi\'{c} \cite{HI}
removed the factor $X^{\varepsilon}$ of Deligne's result
and obtained the bound $O_f(X^{1/3})$.
The current best record is $O_f\left(X^{1/3}(\log X)^{-0.1185}\right)$ due to Wu \cite{Wu}.
For the case of Maass cusp forms, there are fewer results. Assuming the Ramanujan-Petersson conjecture, one also has
the estimate (see for example \cite{FI})
\bea\label{depend on f}
\sum_{n\leq X}\lambda_f(n)\ll_{f,\varepsilon} X^{1/3+\varepsilon}.
\eea
Other interesting results without assuming the Ramanujan-Petersson conjecture
but weaker than \eqref{depend on f} can be found in
Hafner and Ivi\'{c} \cite{HI}, L\"{u} \cite{L1},
Jiang and L\"{u} \cite{JL} and some references therein.

Notice that all the above mentioned estimates depend on the form $f$.
So the second purpose of this paper is to strengthen the estimate in \eqref{depend on f}
by making the dependence on the form $f$ explicit. Our result is the following.
\begin{theorem}\label{coefficient sum}
Let $\lambda_f(n)$ be the normalized Fourier coefficients
of a Hecke-Maass cusp form for $\rm SL_2(\mathbb{Z})$
with Laplacian eigenvalue $\lambda_f(\Delta)=1/4+\mu^2$.
Under the Ramanujan conjecture, we have, for any $\varepsilon>0$,
\[
\sum_{n \leq X}\lambda_f(n)\ll X^{1/3+\varepsilon}\lambda_f(\Delta)^{4/9+\varepsilon},
\]
where the implied constant depends only on $\varepsilon$.
\end{theorem}

\medskip

\noindent{\bf Notation.}
Throughout the paper, $\varepsilon$ and $A$ are arbitrarily small and arbitrarily large positive
numbers, respectively,  which may be different at each occurrence. As usual,
$e(x)=e^{2\pi ix}$ and  the symbol $n\sim X$ means $X<n\leq 2X$.

\section{Preliminaries}
\subsection{Maass cusp forms for $\rm GL_2$}
Let $f$ be a Hecke-Maass cusp form for $\rm SL_2(\mathbb{Z})$
with Laplace eigenvalue $1/4+\mu^2$, with the normalized Fourier coefficients $\lambda_f(n)$.
For $\Re(s)>1$, the $L$-function associated to $f$ is given by
\bea\label{Dirichlet series}
L(s,f)=\sum_{n=1}^{\infty}\frac{\lambda_f(n)}{n^{s}},
\eea
which satisfies the functional equation
\bea\label{functional equation}
L(1-s,f)=(-1)^{\eta}\gamma(s)L(s,f),
\eea
where $\eta=0$ or 1 according as $f$ is even or odd, and
\bea\label{Gamma}
\gamma(s)=\pi^{1-2s}\prod_{\pm }
\Gamma\left(\frac{s+\eta\pm i\mu}{2}\right)\Gamma\left(\frac{1-s+\eta\pm i\mu}{2}\right)^{-1}.
\eea
\subsection{Summation formulas}
We first recall the Poisson summation formula over an arithmetic progression.
\begin{lemma}Let $\beta\in\mathbb{Z}$
and $c\in\mathbb{Z}_{\geq1}$. For a Schwartz function $f:\mathbb{R}\rightarrow\mathbb{C}$, we have
\bna
\underset{\begin{subarray}{c}n\in\mathbb{Z}
\\ n\equiv \beta\bmod c\end{subarray}}{\sum}f(n)
=\frac{1}{c}\sum_{n\in\mathbb{Z}}
\widehat{f}\left(\frac{n}{c}\right)
e\left(\frac{n\beta}{c}\right),
\ena
where $\widehat{f}(y)=\int_{\mathbb{R}}f(x)e(-xy)\mathrm{d}x$
is the Fourier transform of $f$.
\end{lemma}
\begin{proof}
See e.g. \cite[Eq.(4.24)]{IK}.
\end{proof}
We have the following
Voronoi formula for $\rm SL_2(\mathbb{Z})$ (see \cite[Eqs. (1.12), (1.15)]{MS}).

\begin{lemma}\label{GL2 Voronoi formula}
Let $\varphi(x)$ be a smooth function compactly supported on $\mathbb{R^+}$.
Let $a, \overline{a}, c\in\mathbb{Z}$ with $c\neq0, (a,c)=1$ and
$a\overline{a}\equiv1\;(\text{{\rm mod }} c)$. Then
\bna
\sum_{m=1}^{\infty}\lambda_f(m)e\left(\frac{am}{c}\right)\varphi(m)
=c\sum_\pm\sum_{m=1}^{\infty}\frac{\lambda_f(m)}{m}e\left(\pm\frac{\overline{a}m}{c}\right)
\Psi_\pm\left(\frac{m}{c^2}\right),
\ena
where for $\sigma>-1$,
\bea\label{GL2 integral-1}
\Psi_\pm(x)=\frac{1}{4\pi^2 i}\int_{(\sigma)}(\pi^2 x)^{-s}\rho_f^{\pm}(s)\widetilde{\varphi}(-s) \mathrm{d}s,
\eea
with
\bea\label{gamma}
\rho_f^{\pm}(s)=\prod_\pm\frac{\Gamma(\frac{1+s\pm i\mu}{2})}
{\Gamma(\frac{-s\pm i\mu}{2})}\pm
\prod_\pm\frac{\Gamma(\frac{2+s\pm i\mu}{2})}
{\Gamma(\frac{1-s\pm i\mu}{2})}.
\eea
Here $\widetilde{\varphi}(s)=\int_0^{\infty}\varphi(u)u^{s-1}\mathrm{d}u$ is the Mellin transform of $\varphi$.
\end{lemma}
\subsection{Stirling's formula}
By Stirling asymptotic formula (see \cite[Section 8.4, Eq. (4.03)]{O}),
for $|\arg s|\leq \pi-\varepsilon$, $|s|\gg 1$ and any $\varepsilon>0$,
\bna
\ln \Gamma(s)=\left(s-\frac{1}{2}\right)\ln s-s+\frac{1}{2}\ln(2\pi)
+\sum_{j=1}^{K_1}\frac{B_{2j}}{2j(2j-1)s^{2j-1}}+O_{K_1,\varepsilon}\left(\frac{1}{|s|^{2K_1+1}}\right),
\ena
where $B_j$ are Bernoulli numbers. Thus for $s=\sigma+i\tau$, $\sigma$ fixed and $|t|\geq 2$,
\bea\label{Stirling approximation}
\Gamma(\sigma+i\tau)=\sqrt{2\pi}(i\tau)^{\sigma-1/2}e^{-\pi|\tau|/2}\left(\frac{|\tau|}{e}\right)^{i\tau}
\left(1+\sum_{j=1}^{K_2}\frac{c_j}{\tau^j}+O_{\sigma,K_2,\varepsilon}
\bigg(\frac{1}{|\tau|^{K_2+1}}\bigg)\right),
\eea
where the constants $c_j$ depend on $j,\sigma$ and $\varepsilon$.
Thus for $\sigma\geq -1/2$,
\bea\label{Gamma-2}
\rho_f^{\pm}(\sigma+i\tau)&=&
\prod_\pm\frac{\Gamma(\frac{1+\sigma+i(\tau\pm\mu)}{2})}
{\Gamma(\frac{-\sigma-i(\tau\pm\mu)}{2})}\pm
\prod_\pm\frac{\Gamma(\frac{2+\sigma+i(\tau\pm\mu)}{2})}
{\Gamma(\frac{1-\sigma-i(\tau\pm\mu)}{2})}\nonumber\\
&\ll &(|\tau+\mu||\tau-\mu|)^{\sigma+1/2}.
\eea
\subsection{Estimates for exponential integrals}

Let
\begin{equation*}
 I = \int_{\mathbb{R}} w(y) e^{i \varrho(y)} dy.
\end{equation*}
We need the following evaluation for exponential integrals
which are
 Lemma 8.1 and Proposition 8.2 of \cite{BKY} in the language of inert functions
 (see \cite[Lemma 3.1]{KPY}).

Let $\mathcal{F}$ be an index set, $Y: \mathcal{F}\rightarrow\mathbb{R}_{\geq 1}$ and under this map
$T\mapsto Y_T$
be a function of $T \in \mathcal{F}$.
A family $\{w_T\}_{T\in \mathcal{F}}$ of smooth
functions supported on a product of dyadic intervals in $\mathbb{R}_{>0}^d$
is called $Y$-inert if for each $j=(j_1,\ldots,j_d) \in \mathbb{Z}_{\geq 0}^d$
we have
\bna
C(j_1,\ldots,j_d)
= \sup_{T \in \mathcal{F} } \sup_{(y_1, \ldots, y_d) \in \mathbb{R}_{>0}^d}
Y_T^{-j_1- \cdots -j_d}\left| y_1^{j_1} \cdots y_d^{j_d}
w_T^{(j_1,\ldots,j_d)}(y_1,\ldots,y_d) \right| < \infty.
\ena

\begin{lemma}
\label{lemma:exponentialintegral}
 Suppose that $w = w_T(y)$ is a family of $Y$-inert functions,
 with compact support on $[Z, 2Z]$, so that
$w^{(j)}(y) \ll (Z/Y)^{-j}$.  Also suppose that $\varrho$ is
smooth and satisfies $\varrho^{(j)}(y) \ll H/Z^j$ for some
$H/Y^2 \geq R \geq 1$ and all $y$ in the support of $w$.
\begin{enumerate}
 \item
 If $|\varrho'(y)| \gg H/Z$ for all $y$ in the support of $w$, then
 $I \ll_A Z R^{-A}$ for $A$ arbitrarily large.
 \item If $\varrho''(y) \gg H/Z^2$ for all $y$ in the support of $w$,
 and there exists $y_0 \in \mathbb{R}$ such that $\varrho'(y_0) = 0$ (note $y_0$ is
 necessarily unique), then
 \begin{equation*}
  I = \frac{e^{i \varrho(y_0)}}{\sqrt{\varrho''(y_0)}}
 F(y_0) + O_{A}(  Z R^{-A}),
 \end{equation*}
where $F(y_0)$ is an $Y$-inert function (depending on $A$)  supported
on $y_0 \asymp Z$.
\end{enumerate}
\end{lemma}

We also need the
second derivative test (see \cite[Lemma 5.1.3]{Hux}).
	
\begin{lemma}\label{lem: 2st derivative test, dim 1}
Let $\varrho(x)$ be real and twice
differentiable on the open interval $[a, b]$
with $ \varrho'' (x) \gg \lambda_0>0$  on $[a, b]$. Let $w(x)$
be real on $[ a, b]$ and let $V_0$ be its total
variation on $[ a, b]$ plus the maximum modulus of $w(x)$ on $[ a, b]$.
Then
		\begin{align*}
	I\ll \frac {V_0} {\sqrt{\lambda_0}}.
		\end{align*}
	\end{lemma}

\subsection{The circle method}

Let $\delta: \mathbb{Z}\rightarrow \{0,1\}$ be defined as
$\delta(0)=1$ and $\delta(n)=0$ for $n\neq 0$.
We will use a version of the $\delta$-method by
Duke, Friedlander and Iwaniec (see \cite[Chapter 20]{IK})
which states that for any $n\in \mathbb{Z}$ and $C\in \mathbb{R}^+$, we have
\bea\label{DFI's}
\delta(n)=\frac{1}{C}\sum_{1\leq c\leq C} \;\frac{1}{c}\;
\sideset{}{^\star}\sum_{a\bmod{c}}e\left(\frac{na}{c}\right)
\int_\mathbb{R}g(c,\zeta) e\left(\frac{n\zeta}{cC}\right)\mathrm{d}\zeta,
\eea
where the $\star$ on the sum indicates
that the sum over $a$ is restricted to $(a,c)=1$.
The function $g$ has the following properties
(see \cite[(20.158), (20.159)]{IK} and  \cite[Lemma 15]{HB})
\bea\label{g}
g(c,\zeta)\ll |\zeta|^{-A}, \;\;\;\;\;\; g(c,\zeta) =1+
O\left(\frac{C}{c}\left(\frac{c}{C}+|\zeta|\right)^A\right)
\eea
for any $A>1$ and
\bna
\frac{\partial^j}{\partial \zeta^j}g(c,\zeta)\ll
|\zeta|^{-j}\min\left(|\zeta|^{-1},\frac{C}{c}\right)\log C, \quad j\geq 1.
\ena
\subsection{Some estimates}
We quote the following results from
Pitt (see \cite[Section 2]{Pitt}),
Tolev (see \cite[Section 2]{T})
and Liu and Ren (see \cite[Lemma 3.3]{LR}) respectively.
\begin{lemma}\label{Pitt}
Let $c_1$ be the largest square-free factor of $c\in \mathbb{N}$ such that
$c=c_1c_2$, $(c_1,c_2)=1$. Then for any $\varepsilon>0$ and $\beta\in \mathbb{R}$,
we have
\bna
\sum_{n\sim X}S(m,n,c)e(\alpha n^2+\beta n)\ll (Xc)^{1/2+\varepsilon}+\mathbf{T}^{1/2},
\ena
where
\bea\label{sum T}
\mathbf{T}=(m,c)^{1/2}\tau^2(c_1)c_1^{1/2}c_2\sum_{c=c_3c_4}c_4^{1/2}
\sum_{u\bmod c_3\atop (u,c_3)=1}\sum_{1\leq h<X}\min\left\{X,\left(2\left\|2\alpha h+\frac{u}{c_3}\right\|\right)^{-1}\right\}.
\eea
Here $\|x\|$ denote the distance from $x$ to the nearest integer.
\end{lemma}
\begin{lemma}\label{new series}
Let $M\geq 2$. Then for any $\varepsilon>0$, there exists a smooth function
$G(M, x)$,  which is periodic with period one and satisfies
\begin{equation*}
\begin{split}
\min(M,\|x\|^{-1})\leq G(M,x).
\end{split}
\end{equation*}
Moreover, $G(M,x)$ has a Fourier expansion
\begin{equation*}
\begin{split}
G(M,x)=\sum_{n}b(n)e(nx)
\end{split}
\end{equation*}
with coefficients satisfying $b(n)\ll\log M$ and
\begin{equation*}
\begin{split}
\sum_{|n|>M^{1+\varepsilon}}|b(n)|\ll_{A,\varepsilon}M^{-A}
\end{split}
\end{equation*}
for any constant $A>0$.
\end{lemma}

\begin{lemma}\label{LR lemma}
Let $f(n)$ denote the largest square-free factor of $n$ such that $(f(n),n/f(n))=1$. Then we have
\begin{equation*}
\begin{split}
\sum_{n\leq X}f^{-1/4}(n)\ll X^{3/4}
\end{split}
\end{equation*}
and
\begin{equation*}
\begin{split}
\sum_{n\leq X}f^{-1/2}(n)\ll X^{1/2}\log X.
\end{split}
\end{equation*}
\end{lemma}
We also need the following estimate (see Karatsuba \cite[Chapter \uppercase\expandafter{\romannumeral6}, \S2, Lemma 5]{K}).
\begin{lemma}\label{x series}
Let
\bna
\alpha=\frac{\ell}{q}+\frac{\theta}{q^2},\quad (\ell,q)=1, \quad q\geq1, \quad|\theta|\leq1.
\ena
Then for any $\beta\in \mathbb{R}$, $U>0$ and $P\geq1$, we have
\bna
\sum_{x=1}^P
\min\left\{U,\left\|\alpha x+\beta\right\|^{-1}\right\}
\leq6\left(\frac{P}{q}+1\right)(U+q\log q).
\ena
\end{lemma}

\section{Proof of Theorem \ref{the quadratic exponential sums}}\label{section 1}
We assume $\mu>X^{\varepsilon}$, otherwise Theorem \ref{the quadratic exponential sums} follows from Liu and Ren \cite{LR}.
By dyadic subdivision it suffices to prove the required estimate for the sum
\bna
S(X,\alpha,\beta)=\sum_{n\sim X}\lambda_f(n)
e\left(\alpha n^2+\beta n\right).
\ena

Let $V(x)\in \mathcal {C}_c^{\infty}(3/4,9/4)$ be identically one
on $[1,2]$ with derivatives satisfying $V^{(j)}(x)\ll_j 1$ for any integer $j\geq 0$.
Then we can write $S(X,\alpha,\beta)$ as
\bea\label{aim-sum}
S(X,\alpha,\beta)=\sum_{n\sim X}e\left(\alpha n^2+\beta n\right)
\sum_{m=1}^{\infty}\lambda_f(m)
V\left(\frac{m}{X}\right)\delta(n-m),
\eea
where $\delta(n)=\left\{
\begin{aligned}
&1\, \textit{ if } n=0,\\
&0\, \textit{ if } n\neq 0
\end{aligned} \right. $ is the Kronecker delta function.

Plugging the identity \eqref{DFI's} for
$\delta(n)$ into \eqref{aim-sum} and
exchanging the order of integration and summations,
we get
\bna
S(X,\alpha,\beta)&=&\frac{1}{C}
\int_{\mathbb{R}} \sum_{1\leq c\leq C}\frac{g(c,\zeta)}{c}
\sum_{n\sim X}e\left(\alpha n^2+\beta n+\frac{n\zeta}{cC}\right)\;
\\&&\times \sideset{}{^\star}\sum_{a\bmod{c}}
e\left(\frac{na}{c}\right)
\sum_{m=1}^{\infty}\lambda_{f}(m)e\left(-\frac{ma}{c}\right)
V\left(\frac{m}{X}\right)
e\left(-\frac{m\zeta}{cC}\right)\mathrm{d}\zeta,
\ena
where $C>1$ is a parameter to be chosen later.
Note that the contribution from $|\zeta|\leq X^{-G}$ for $G>0$ sufficiently large
is negligible.
Moreover, by the first property in \eqref{g}, we can restrict $\zeta$ in the range
$|\zeta|\leq X^{\varepsilon}$ up to an negligible error. So we can
insert a smooth partition of unity for the $\zeta$-integral and write $S(X,\alpha,\beta)$ as
\bna
S(X,\alpha,\beta)&=&\sum_{X^{-G}\ll \Xi\ll X^{\varepsilon}\atop \text{dyadic}}\frac{1}{C}
\int_{\mathbb{R}} \varpi\left(\frac{\zeta}{\Xi}\right) \sum_{1\leq c\leq C}\frac{g(c,\zeta)}{c}
\sum_{n\sim X}e\left(\alpha n^2+\beta n+\frac{n\zeta}{cC}\right)\;
\\&&\times \sideset{}{^\star}\sum_{a\bmod{c}}
e\left(\frac{na}{c}\right)
\sum_{m=1}^{\infty}\lambda_{f}(m)e\left(-\frac{ma}{c}\right)
V\left(\frac{m}{X}\right)
e\left(-\frac{m\zeta}{cC}\right)\mathrm{d}\zeta+O_A(X^{-A}),
\ena
where $\varpi(x)\in \mathcal{C}_c^{\infty}(1,2)$
 satisfying $\varpi^{(j)}(x)\ll_j 1$ for any integer $j\geq 0$.
 Without loss of generality,
we only consider the contribution from $\zeta>0$ (the proof for $\zeta<0$ is entirely
similar). By abuse of notation, we still write the contribution from $\zeta>0$
as $S(X,\alpha,\beta)$.

Next we break the $c$-sum $\sum_{1\leq c\leq C}$ into dyadic segments
$c\sim C_0$ with $1\ll C_0\ll C$ and write
\bea\label{C range}
S(X,\alpha,\beta)=\sum_{X^{-G}\ll \Xi\ll X^{\varepsilon}\atop \text{dyadic}}
\sum_{1\ll C_0\ll C\atop \text{dyadic}}S(X,\alpha,\beta,C_0,\Xi)+O_A(X^{-A})
\eea
with
\bea\label{beforeVoronoi}
S(X,\alpha,\beta,C_0,\Xi)&=&\frac{1}{C}
\int_{\mathbb{R}} \varpi\left(\frac{\zeta}{\Xi}\right)\sum_{c\sim C_0}\frac{g(c,\zeta)}{c}
\sum_{n\sim X}e\left(\alpha n^2+\beta n+\frac{n\zeta}{cC}\right)\;
\nonumber\\&&\times \sideset{}{^\star}\sum_{a\bmod{c}}
e\left(\frac{na}{c}\right)
\sum_{m=1}^{\infty}\lambda_{f}(m)e\left(-\frac{ma}{c}\right)
V\left(\frac{m}{X}\right)
e\left(-\frac{m\zeta}{cC}\right)\mathrm{d}\zeta.
\eea

We now proceed to estimate $S(X,\alpha,\beta,C_0,\Xi)$ for $1\ll C_0\ll C$.
Applying Lemma \ref{GL2 Voronoi formula} with $\varphi(x)=V\left(x/X\right)
e\left(-x\zeta/(cC)\right)$ to transform the sum over $m$ we get
\bea\label{Voronoi}
\sum_{m=1}^{\infty}\lambda_{f}(m)e\left(-\frac{ma}{c}\right)
V\left(\frac{m}{X}\right)
e\left(-\frac{m\zeta}{cC}\right)
=c\sum_\pm\sum_{m=1}^\infty\frac{\lambda_f(m)}{m}e\left(\mp\frac{\overline{a}m}{c}\right)
\Psi^\pm\left(\frac{m}{c^2},c,\zeta\right),
\eea
where by \eqref{GL2 integral-1},
\bea\label{sigma}
\Psi^\pm\left(x,c,\zeta\right)=\frac{1}{4\pi^2}\int_{\mathbb{R}}(\pi^2 xX)^{-\sigma-i\tau}
\rho_f^{\pm}(\sigma+i\tau)V^{\dagger}\left(\frac{\zeta X}{cC},-\sigma-i\tau\right) \mathrm{d}\tau
\eea
with $\rho_f^{\pm}(s)$ defined in \eqref{gamma} and
\bna
V^\dagger(r,s)=\int_0^\infty V(x)e(-rx)x^{s-1}\mathrm{d}x.
\ena
Plugging \eqref{Voronoi} into \eqref{beforeVoronoi}, we obtain
\bea\label{afterVoronoi}
S(X,\alpha,\beta,C_0,\Xi)&=&\frac{1}{C}\sum_\pm
\int_{\mathbb{R}} \varpi\left(\frac{\zeta}{\Xi}\right)\sum_{c\sim C_0}g(c,\zeta)
\sum_{m=1}^\infty\frac{\lambda_f(m)}{m}
\Psi^\pm\left(\frac{m}{c^2},c,\zeta\right)\nonumber\\&&\times
\sum_{n\sim X}e\left(\alpha n^2+\beta n+\frac{n\zeta}{cC}\right)S(n,\mp m;c)
\mathrm{d}\zeta,
\eea
where $S(n,m;c)$ is the classical Kloosterman sum.

The integral $\Psi^\pm\left(x,c,\zeta\right)$ has the following properties.
\begin{lemma}\label{integral-lemma-1}
Let $r=\zeta X/cC$ and $\zeta\asymp \Xi$.

(1) Suppose $X\Xi/(cC)\gg X^{\varepsilon}$. Then for
$\mu^{1-\varepsilon}\ll r\ll \mu^{1+\varepsilon}$,
$\Psi^\pm\left(x,c,\zeta\right)=\mathbf{\Psi}_1+\mathbf{\Psi}_2$, where
$\mathbf{\Psi}_1$
is negligibly small unless $xX\ll \mu^{1+\varepsilon}$, in which case
$$
\mathbf{\Psi}_1\ll (rXx)^{1/2},
$$
and $\mathbf{\Psi}_2$ is negligibly small unless $xX\ll r\mu^{1+\varepsilon}$, in which case
$$
\mathbf{\Psi}_2\ll (xX)^{1/2}.
$$

(2) Suppose $X\Xi/(cC)\gg X^{\varepsilon}$. Then
for $r\ll \mu^{1-\varepsilon}$ or $r\gg \mu^{1+\varepsilon}$, $\Psi^\pm\left(x,c,\zeta\right)$ is negligibly small
unless $x\asymp \max\{r^2,\mu^2\}/X$, in which case
\bna
\Psi^{\pm}(x,c,\zeta)\ll (xX)^{1/2}.
\ena

(3) If $X\Xi/(cC)\ll X^{\varepsilon}$, then $\Psi^\pm\left(x,c,\zeta\right)$
is negligibly small unless $xX\ll \mu^{2+\varepsilon}$, in which case
\bna
\Psi^\pm\left(x,c,\zeta\right)\ll (xX)^{1/2+\varepsilon}.
\ena
\end{lemma}

\begin{proof}
For the case  $r\asymp X\Xi/(cC)\gg X^{\varepsilon}$,
we apply the stationary phase to
the integral $V^{\dagger}\left(r,-\sigma-i\tau\right)$. Write
\bna
V^{\dagger}\left(r,-\sigma-i\tau\right)=
\int_0^{\infty}
V\left(u\right)u^{-\sigma-1}\exp\left(i\varrho(u)\right)
\mathrm{d}u,
\ena
where $\varrho(u)=-2\pi ru-\tau\log u$.
Note that
\bna
\varrho'(u)&=&-2\pi r-\tau/u,\\
\varrho^{(j)}(u)
&=&\tau (-1)^j(j-1)! u^{-j}\asymp |\tau|, \qquad j=2,3,\ldots.
\ena
By repeated integration by parts one shows that $V^{\dagger}\left(r,-\sigma-i\tau\right)$
is negligibly small unless $|\tau| \asymp r$ and $\tau<0$ (note that $r>0$ here).
The stationary point is $u_0=-\tau/(2\pi r)$.
Applying Lemma \ref{lemma:exponentialintegral} (2) with $Y=Z=1$ and
$H=R=\tau\gg X^{\varepsilon}$, we have
\bea\label{asymptotic formula}
V^{\dagger}\left(r,-\sigma-i\tau\right)
=\tau^{-1/2}
V^{\natural}_{\sigma}\left(\frac{-\tau}{2\pi r}\right)
e\left(-\frac{\tau}{2\pi }\log\frac{-\tau}{2\pi e r}\right)+O_A\left(X^{-A}\right),
\eea
where $V_{\sigma}^{\natural}(x)$ is an inert function (depending on $A$ and $\sigma$)
supported on $x\asymp 1$. Plugging \eqref{asymptotic formula} into \eqref{sigma}, we obtain
\bna
\Psi^\pm\left(x,c,\zeta\right)&=&\frac{1}{4\pi^2}\int_{-\infty}^0(\pi^2 xX)^{-\sigma-i\tau}
\rho_f^{\pm}(\sigma+i\tau)\nonumber\\&&\times
\tau^{-1/2}
V^{\natural}_{\sigma}\left(\frac{-\tau}{2\pi r}\right)
e\left(-\frac{\tau}{2\pi }\log\frac{-\tau}{2\pi e r}\right)\mathrm{d}\tau+O_A\left(X^{-A}\right),
\ena
where $r=\zeta X/(cC)>0$.
Making a change of variable $\tau\rightarrow -r\tau$,
\bea\label{The integral:1}
\Psi^\pm\left(x,c,\zeta\right)&=&\frac{\sqrt{-r}}{4\pi^2}\int_0^{\infty}
(\pi^2 xX)^{-\sigma+ir\tau}\rho_f^{\pm}\left(\sigma-ir\tau\right)\nonumber\\&&\times
\tau^{-1/2}V^{\natural}_{\sigma}\left(\frac{\tau}{2\pi}\right)
e\left(\frac{r\tau}{2\pi}\log\frac{\tau}{2\pi e}\right)\mathrm{d}\tau+O_A\left(X^{-A}\right),
\eea
where by \eqref{gamma},
\bea\label{Gamma3}
\rho_f^{\pm}\left(\sigma-ir\tau\right)=
\prod_{\pm}\frac{\Gamma\left(\frac{1+\sigma-i(r\tau\pm \mu)}{2}\right)}
{\Gamma\left(\frac{-\sigma+i(r\tau\pm \mu)}{2}\right)}\pm
\prod_{\pm}\frac{\Gamma\left(\frac{2+\sigma-i(r\tau\pm \mu)}{2}\right)}
{\Gamma\left(\frac{1-\sigma+i(r\tau\pm \mu)}{2}\right)}.
\eea

(1) For $\mu^{1-\varepsilon}\ll r\ll\mu^{1+\varepsilon}$, we divide the range of
$\tau$ into two pieces: $$(0,\infty)=\{\tau||r\tau-\mu|\leq \mu^{\varepsilon}\}
\cup \{\tau||r\tau-\mu|>\mu^{\varepsilon}\}:=\mathbf{I}_1+\mathbf{I}_2$$
and correspondingly denote by the integral over $\mathbf{I}_j$ by
$\mathbf{\Psi}_j$, $j=1,2$. Then by \eqref{Gamma-2},
\bna
\mathbf{\Psi}_1&\ll&r^{1/2}(xX)^{-\sigma}\int_{\mathbf{I}_1}
(|r\tau-\mu||r\tau+\mu|)^{\sigma+1/2}
\left|V^{\natural}_{\sigma}\left(\frac{\tau}{2\pi}\right)\right|\mathrm{d}\tau\\
&\ll&r^{1/2}\mu^{1/2+\varepsilon}(xX/\mu^{1+\varepsilon})^{-\sigma}.
\ena
By taking $\sigma$ sufficiently large, one sees that $\mathbf{\Psi}_1$
is negligibly small unless $xX\ll \mu^{1+\varepsilon}$,
in which case by taking $\sigma=-1/2$ we have
the estimate
\bea\label{p-1}
\mathbf{\Psi}_1\ll (rxX)^{1/2}.
\eea

For $\tau\in \mathbf{I}_2$, by \eqref{Gamma3} and Stirling's approximation
in \eqref{Stirling approximation},
we have
\bea\label{Gamma4}
&&\rho_f^{\pm}\left(\sigma-ir\tau\right)=\left(\prod_{\pm}
\left(\frac{|r\tau\pm \mu|}{2e}\right)^{-i(r\tau\pm \mu)}
|r\tau\pm \mu|^{\sigma+1/2}\right)\nonumber
\\&&\times
\big(h_{\sigma,1}(r\tau+\mu)h_{\sigma,1}(r\tau-\mu)
\pm h_{\sigma,2}(r\tau+\mu)h_{\sigma,2}(r\tau-\mu)\big)+
O_{\sigma,A}\big(X^{-A}\big),
\eea
where $h_{\sigma,j}(x)$, $j=1,2$, satisfy $h_{\sigma,j}(x)\ll_{\sigma,j,A} 1$ and
$
x^{\ell}h_{\sigma,j}^{(\ell)}(x)\ll_{\sigma,j,\ell,A} x^{-1}
$
for any integer $\ell\geq 1$.
Then by \eqref{The integral:1} and \eqref{Gamma4},
\bea\label{The integral:20}
\mathbf{\Psi}_2&=&\frac{\sqrt{-r}}{4\pi^2}\int_0^{\infty}
(\pi^2 xX)^{-\sigma+ir\tau}\left(\prod_{\pm}
\left(\frac{|r\tau\pm \mu|}{2e}\right)^{-i(r\tau\pm \mu)}
|r\tau\pm \mu|^{\sigma+1/2}\right)\nonumber
\\&&\times
\big(h_{\sigma,1}(r\tau+\mu)h_{\sigma,1}(r\tau-\mu)
\pm h_{\sigma,2}(r\tau+\mu)h_{\sigma,2}(r\tau-\mu)\big)\nonumber\\&&\times
\tau^{-1/2}V^{\natural}_{\sigma}\left(\frac{\tau}{2\pi}\right)
e\left(\frac{r\tau}{2\pi}\log\frac{\tau}{2\pi e}\right)\mathrm{d}\tau+\mathbf{\Psi}_3,
\eea
where
\bna
\mathbf{\Psi}_3\ll r^{1/2}(xX)^{-\sigma}\int_{\mathbf{I}_1}
(|r\tau-\mu||r\tau+\mu|)^{\sigma+1/2}\left|V^{\natural}_{\sigma}
\left(\frac{\tau}{2\pi}\right)\right|\mathrm{d}\tau
\ll r^{1/2}\mu^{1/2+\varepsilon}(xX/\mu^{1+\varepsilon})^{-\sigma},
\ena
which can be negligibly small unless $xX\ll \mu^{1+\varepsilon}$,
in which case by taking $\sigma=-1/2$ we have
\bea\label{p-2}
\mathbf{\Psi}_3\ll (rxX)^{1/2}.
\eea
Denote the first term in \eqref{The integral:20} by $\mathbf{\Psi}_2^0$. Then
\bna
\mathbf{\Psi}_2^0\ll r^{1/2}(xX)^{-\sigma}\int_{\tau\asymp 1}(|r\tau-\mu||r\tau+\mu|)^{\sigma+1/2}
\mathrm{d}\tau
\ll r\mu^{1/2+\varepsilon}\left(\frac{xX}{r\mu^{1+\varepsilon}}\right)^{-\sigma},
\ena
which can be negligibly small unless $xX\ll r\mu^{1+\varepsilon}$,
in which case by taking $\sigma=-1/2$,
\bna
\mathbf{\Psi}_2^0=(-r)^{1/2}(xX)^{1/2}\int_0^{\infty}G(\tau)
\exp\left(i\eta(\tau)\right)\mathrm{d}\tau,
\ena
where, temporarily,
\bna
G(\tau)=\frac{1}{4\pi\sqrt{\tau}}V^{\natural}_{\sigma}\left(\frac{\tau}{2\pi}\right)
\big(h_{\sigma,1}(r\tau+\mu)h_{\sigma,1}(r\tau-\mu)
\pm h_{\sigma,2}(r\tau+\mu)h_{\sigma,2}(r\tau-\mu)\big)
\ena
with $\sigma=-1/2$, and
\bna
\eta(\tau)=r\tau\log \frac{\pi xX}{2e}-(r\tau+\mu)\log\frac{|r\tau+\mu|}{2e}
-(r\tau-\mu)\log\frac{|r\tau-\mu|}{2e}+r\tau\log \tau.
\ena
Note that
\bna
\eta'(\tau)&=&-r\log\frac{|r\tau-\mu||r\tau+\mu|}{2\pi xX\tau},\\
\eta''(\tau)&=&-r\left(\frac{1}{\tau-\mu/r}
+\frac{1}{\tau+\mu/r}-\frac{1}{\tau}\right)
\ena
and
\bna
\int_{|r\tau-\mu|>\sqrt{r} }
\left|\frac{\mathrm{d}G(\tau)}{\mathrm{d}\tau}\right|\mathrm{d}\tau
\ll \max_{\tau\asymp 1}\left\{
\frac{r}{|r\tau-\mu|^2}
,\frac{r}{|r\tau+\mu|^2},1\right\}\ll 1.
\ena
Moreover, for $|r\tau-\mu|>\sqrt{r}$ and $\mu^{1-\varepsilon}\ll r\ll \mu^{1+\varepsilon}$,
\bna
\eta''(\tau)\asymp r\max\limits_{\tau\asymp 1}|\tau-\mu/r|^{-1}.
\ena
Then by Lemma \ref{lem: 2st derivative test, dim 1},
\bna
&&(-r)^{1/2}(xX)^{1/2}\int_{|r\tau-\mu|>\sqrt{r}}G(\tau)
\exp\left(i\eta(\tau)\right)\mathrm{d}\tau\\
&\ll& (xX)^{1/2}\min\limits_{|r\tau-\mu|>\sqrt{r},\tau\asymp 1}
|\tau-\mu/r|^{1/2}\\
&\ll&(xX)^{1/2}.
\ena
Trivially, we have
\bna
(-r)^{1/2}(xX)^{1/2}\int_{|r\tau-\mu|\leq \sqrt{r} }G(\tau)
\exp\left(i\eta(\tau)\right)\mathrm{d}\tau\ll (xX)^{1/2}.
\ena
Assembling the above results, we conclude that
\bea\label{p-3}
\mathbf{\Psi}_2^0\ll (xX)^{1/2}.
\eea
Then the first statement follows from
\eqref{p-1} and \eqref{The integral:20}--\eqref{p-3}.

(2) For $r\ll\mu^{1-\varepsilon}$, we
take $\sigma=-1/2$ in \eqref{The integral:1} to get
\bea\label{The integral:2}
\Psi^\pm\left(x,c,\zeta\right)&=&\frac{\sqrt{-r}}{4\pi^2}\int_0^{\infty}
(\pi^2 xX)^{1/2+ir\tau}\rho_f^{\pm}\left(-\frac{1}{2}-ir\tau\right)\nonumber\\&&\times
\tau^{-1/2}V^{\natural}\left(\frac{\tau}{2\pi}\right)
e\left(\frac{r\tau}{2\pi}\log\frac{\tau}{2\pi e}\right)\mathrm{d}\tau+O_A\left(X^{-A}\right),
\eea
where $V^{\natural}(x)=V^{\natural}_{-1/2}(x)$ and by \eqref{Gamma3},
\bna
\rho_f^{\pm}\left(-\frac{1}{2}-ir\tau\right)=
\prod_{\pm}\frac{\Gamma\left(\frac{1/2-i(r\tau\pm \mu)}{2}\right)}
{\Gamma\left(\frac{1/2+i(r\tau\pm \mu)}{2}\right)}\pm
\prod_{\pm}\frac{\Gamma\left(\frac{3/2-i(r\tau\pm \mu)}{2}\right)}
{\Gamma\left(\frac{3/2+i(r\tau\pm \mu)}{2}\right)}.
\ena
Since $r\ll\mu^{1-\varepsilon}$, using Stirling's approximation
in \eqref{Stirling approximation}, we derive
\bea\label{gamma-22}
&&\rho_f^{\pm}\left(-\frac{1}{2}-ir\tau\right)
=\left(\frac{\mu-r\tau}{2e}\right)^{-i(r\tau-\mu)}
\left(\frac{r\tau+\mu}{2e}\right)^{-i(r\tau+\mu)}\nonumber\\
&&\qquad\times
\big(h_{1}(r\tau-\mu)h_1(r\tau+\mu)\pm h_2(r\tau-\mu)h_2(r\tau+\mu)\big)+
O_{A}\big(X^{-A}\big),
\eea
where $h_j(x)$, $j=1,2$, satisfy $h_j(x)\ll_j1$ and
$
x^{\ell}h_{j}^{(\ell)}(x)\ll_{j,\ell,A} x^{-1}
$
for any integer $\ell\geq 1$. Plugging \eqref{gamma-22} into \eqref{The integral:2}, one has
\bea\label{The integral:3}
\Psi^\pm\left(x,c,\zeta\right)=\frac{(-rxX)^{1/2}}{4\pi}\int_0^{\infty}
V_0^{\pm}(\tau)\exp\left(i\varrho_0(\tau)\right)\mathrm{d}\tau+O_A\left(X^{-A}\right),
\eea
where
\bna
V_0^{\pm}(\tau)=\frac{1}{\sqrt{\tau}}V^{\natural}\left(\frac{\tau}{2\pi}\right)
\big(h_{1}(r\tau-\mu)h_1(r\tau+\mu)\pm h_2(r\tau-\mu)h_2(r\tau+\mu)\big)
\ena
satisfying $\mathrm{d}^{\ell}V_0^{\pm}(\tau)/\mathrm{d}\tau^{\ell}\ll_{\ell} 1$ for any integer $\ell\geq 0$, and
\bna
\varrho_0(\tau)=r\tau\log \frac{\pi xX}{2e}-(r\tau-\mu)\log\frac{\mu-r\tau}{2e}
-(r\tau+\mu)\log\frac{r\tau+\mu}{2e}+r\tau\log \tau.
\ena
We compute
\bna
\varrho'_0(\tau)&=&-r\log\frac{\mu^2-r^2\tau^2}{2\pi xX\tau},\\
\varrho''_0(\tau)&=&-r\left(\frac{1}{\tau-\mu/r}
+\frac{1}{\tau+\mu/r}-\frac{1}{\tau}\right)\asymp r,\quad j=2,3,\ldots.
\ena
By repeated integration by parts one shows that $\Psi^\pm\left(x,c,\zeta\right)$
is negligibly small unless $xX\asymp \mu^2$.
By the second derivative test in Lemma \ref{lem: 2st derivative test, dim 1}, we have
\bna
\Psi^\pm\left(x,c,\zeta\right)\ll (xX)^{1/2}.
\ena

For $r\gg \mu^{1+\varepsilon}$, the proof is similar as that for the case
$r\ll \mu^{1-\varepsilon}$ and we will be brief.
In this case, the formula \eqref{The integral:3} still holds.
Thus repeated integration by parts shows that $\Psi^\pm\left(x,c,\zeta\right)$
is negligibly small unless $xX\asymp r^2$.
Note that the total variation of $V_0^{\pm}(\tau)$ is bounded by 1 and the second
derivative of the phase function is of size $r$.
By the second derivative test in Lemma \ref{lem: 2st derivative test, dim 1}, we have
\bna
\Psi^\pm\left(x,c,\zeta\right)\ll (xX)^{1/2}.
\ena
This proves the second statement of the lemma.

(3) For $X\Xi/(cC)\ll X^{\varepsilon}$, we have $r\asymp X\Xi/(cC)\ll X^{\varepsilon}$.
By repeated integration by parts, one has (see \cite[Lemma 5]{Munshi})
\bna
V^\dagger(r,\sigma+i\tau)\ll_{\sigma} \min\left\{1,
\left(\frac{1+
|r|}{|\tau|}\right)^{j}\right\}.
\ena
Thus
$$V^{\dagger}\left(\frac{\zeta X}{cC},-\sigma-i\tau\right)\ll \left(\frac{X^{\varepsilon}}{|\tau|}\right)^j,$$
which implies that the contribution from $|\tau|\geq X^{\varepsilon}$
can be arbitrarily small by taking $j$ sufficiently large.
Using \eqref{sigma} and the trivial estimate $V^\dagger(r,\sigma+i\tau)\ll 1$, we have
\bna
\Psi^\pm\left(x,c,\zeta\right)&=&\frac{1}{4\pi^2}
\int_{|\tau|\leq X^{\varepsilon}}(\pi^2 xX)^{-\sigma-i\tau}
\rho_f^{\pm}(\sigma+i\tau)V^{\dagger}\left(\frac{\zeta X}{cC},-\sigma-i\tau\right) \mathrm{d}\tau
+O_A(X^{-A})\\
&\ll&(xX)^{-\sigma}\int_{|\tau|\leq X^{\varepsilon}}(|\tau+\mu||\tau-\mu|)^{\sigma+1/2}\mathrm{d}\tau\\
&\ll&\mu^{1+\varepsilon}\left(xX/\mu^2\right)^{-\sigma},
\ena
which implies that the contribution from $xX\gg \mu^{2+\varepsilon}$ is negligible.
For $xX\ll \mu^{2+\varepsilon}$, we shift the line of integration in \eqref{sigma} to
$\sigma=-1/2$ to get
\bna
\Psi^\pm\left(x,c,\zeta\right)\ll (xX)^{1/2+\varepsilon}.
\ena
This finishes the proof of the lemma.
\end{proof}

Now we return to the evaluation of $S(X,\alpha,\beta,C_0,\Xi)$ in \eqref{afterVoronoi}.
Applying Lemma \ref{Pitt}, we have
\bna
S(X,\alpha,\beta,C_0,\Xi)
\ll\sup_{\zeta\asymp\Xi}\frac{X^\varepsilon}{C}\sum_{\pm}\sum_{c\sim C_0}
\sum_{m}\frac{|\lambda_{f}(m)|}{m}\left|\Psi^{\pm}\left(\frac{m}{c^2},c,\zeta\right)\right|
\left((Xc)^{1/2+\varepsilon}+\textbf{T}^{1/2}\right),
\ena
where $\textbf{T}$ is given by \eqref{sum T}.
By Lemma \ref{integral-lemma-1}, we obtain
\bea\label{S(X,alpha,beta,C0,Xi)}
S(X,\alpha,\beta,C_0,\Xi)
&\ll& \textbf{1}_{\frac{X\Xi}{C\mu^{1+\varepsilon}}\ll C_0\ll\frac{X\Xi}{C\mu^{1-\varepsilon}}}\sum_{i=1}^2S_i
+\textbf{1}_{\frac{X\Xi}{C\mu^{1-\varepsilon}}\ll C_0\ll\frac{X^{1-\varepsilon}\Xi}{C}}S_3\nonumber\\
&&+\textbf{1}_{1\leq C_0\ll\frac{X\Xi}{C\mu^{1+\varepsilon}}}S_4
+\textbf{1}_{\frac{X^{1-\varepsilon}\Xi}{C}\ll C_0\ll C}S_5,
\eea
where $\textbf{1}_\mathbf{A}=1$ is $\mathbf{A}$ is true and equals 0 otherwise,
\bna
S_1&=&\frac{X^\varepsilon}{C}\sum_{c\sim C_0}
\sum_{m\ll c^2\mu^{1+\varepsilon}/X}
\frac{|\lambda_{f}(m)|}{m}
\left(\frac{mX^2\Xi}{c^3C}\right)^{1/2}
\left((Xc)^{1/2+\varepsilon}+\textbf{T}^{1/2}\right),\\
S_2&=&\frac{X^\varepsilon}{C}\sum_{c\sim C_0}
\sum_{m\ll c\mu^{1+\varepsilon}\Xi/C}
\frac{|\lambda_{f}(m)|}{m}\left(\frac{mX}{c^2}\right)^{1/2}
\left((Xc)^{1/2+\varepsilon}+\textbf{T}^{1/2}\right),\\
S_3&=&\frac{X^\varepsilon}{C}\sum_{c\sim C_0}
\sum_{m\asymp c^2\mu^{2}/X}
\frac{|\lambda_{f}(m)|}{m}\left(\frac{mX}{c^2}\right)^{1/2}
\left((Xc)^{1/2+\varepsilon}+\textbf{T}^{1/2}\right),\\
S_4&=&\frac{X^\varepsilon}{C}\sum_{c\sim C_0}
\sum_{m\asymp X\Xi^2/C^2}
\frac{|\lambda_{f}(m)|}{m}\left(\frac{mX}{c^2}\right)^{1/2}
\left((Xc)^{1/2+\varepsilon}+\textbf{T}^{1/2}\right)
\ena
and
\bna
S_5&=&\frac{X^\varepsilon}{C}\sum_{c\sim C_0}
\sum_{m\ll c^2\mu^{2+\varepsilon}/X}
\frac{|\lambda_{f}(m)|}{m}\left(\frac{mX}{c^2}\right)^{1/2+\varepsilon}
\left((Xc)^{1/2+\varepsilon}+\textbf{T}^{1/2}\right).
\ena
Obviously, $S_3$ can be dominated by $S_5$.
We use the strategy of Liu and Ren \cite{LR} to deal with $S_1$, and $S_2$--$S_5$ can be estimated similarly.
Firstly, by \eqref{sum T} and Lemma \ref{new series} we have,
\begin{equation*}
\begin{split}
\textbf{T}&\ll(m,c)^{1/2} c_1^{1/2+\varepsilon}c_2\sum_{c=c_3c_4}c_4^{1/2}\sum_{1\leq h< X}
\sum_{u\bmod c_3}G\left(X, 2\alpha h+\frac{u}{c_3}\right)\\
&=(m,c)^{1/2}cc_1^{-1/2+\varepsilon}\sum_{c=c_3c_4}c_4^{1/2}\sum_{1\leq h< X}
\sum_{u\bmod c_3}\left\{b(0)+\sum_{|n|\geq 1}b(n)e\left(2n\alpha h+\frac{nu}{c_3}\right)\right\}\\
&\ll |T_0|+|T_1|+|T_2|+X^{-A},
\end{split}
\end{equation*}
say, where $c_1, c_2$ are defined as in Lemma \ref{Pitt}. Trivially, we have
\bea\label{T0}
T_0&=&(m,c)^{1/2}cc_1^{-1/2+\varepsilon}\sum_{c=c_3c_4}c_4^{1/2}\sum_{1\leq h< X}
\sum_{u\bmod c_3}b(0)\nonumber\\
&\ll &X(m,c)^{1/2}c^{2+\varepsilon}c_1^{-1/2}.
\eea
Moreover,
\bea\label{T1}
T_1&=&(m,c)^{1/2}cc_1^{-1/2+\varepsilon}\sum_{c=c_3c_4}c_4^{1/2}\sum_{1\leq h< X}
\sum_{u\bmod c_3}\sum_{1\leq n \leq X^{1+\varepsilon}}b(n)e\left(2n\alpha h+\frac{nu}{c_3}\right)\nonumber\\
&\ll& (\log X)(m,c)^{1/2}c^2c_1^{-1/2+\varepsilon}\sum_{c=c_3c_4}c_4^{-1/2}\sum_{1\leq k\leq X^{1+\varepsilon}/c_3}
\min\left\{X,\left\|2c_3k\alpha\right\|^{-1}\right\},
\eea
which is based on the orthogonality of additive characters and the elementary estimate
$\sum_{h\leq X}e(\xi h)\ll\min(X, \|\xi\|^{-1})$. Finally,
\begin{equation*}
\begin{split}
T_2&=(m,c)^{1/2}cc_1^{-1/2+\varepsilon}\sum_{c=c_3c_4}c_4^{1/2}\sum_{1\leq h< X}
\sum_{u\bmod c_3}\sum_{1\leq n \leq X^{1+\varepsilon}}b(-n)e\left(-2n\alpha h-\frac{nu}{c_3}\right),
\end{split}
\end{equation*}
which can be estimated similarly as $T_1$.
Hence by Cauchy--Schwartz inequality,
the Rankin--Selberg estimate in \eqref{GL2: Rankin Selberg}
and the above estimates, we have
\begin{equation*}
\begin{split}
S_1
&\ll\frac{X^{1+\varepsilon}\Xi^{1/2}}{C^{3/2}}\sum_{c\sim C_0}\frac{1}{c^{3/2}}
\sum_{m\ll c^2\mu^{1+\varepsilon}/X}\frac{|\lambda_{f}(m)|}{m^{1/2}}
(Xc)^{1/2+\varepsilon}+\frac{X^{1+\varepsilon}\Xi^{1/2}}{C^{3/2}}\sum_{c\sim C_0}\frac{1}{c^{3/2}}
\sum_{m\ll c^2\mu^{1+\varepsilon}/X}\frac{|\lambda_{f}(m)|}{m^{1/2}}\textbf{T}^{1/2}\\
&\ll\frac{X^{3/2+\varepsilon}\Xi^{1/2}}{C^{3/2}}\sum_{c\sim C_0}\frac{1}{c^{1-\varepsilon}}\left(\frac{c^2\mu^{1+\varepsilon}}{X}\right)^{1/2+\varepsilon}
+\frac{X^{1+\varepsilon}\Xi^{1/2}}{C^{3/2}}\sum_{i=0}^2\sum_{c\sim C_0}\frac{1}{c^{3/2}}
\sum_{m\ll c^2\mu^{1+\varepsilon}/X}\frac{|\lambda_{f}(m)|}{m^{1/2}}T_i^{1/2}\\
&\ll\frac{X^{1+\varepsilon}\mu^{1/2+\varepsilon}}{C^{3/2}}C_0^{1+\varepsilon}
+\sum_{i=0}^2S_{1i},
\end{split}
\end{equation*}
recalling $\Xi\ll X^\varepsilon$, where
\begin{equation*}
\begin{split}
S_{1i}&=\frac{X^{1+\varepsilon}}{C^{3/2}}\sum_{c\sim C_0}\frac{1}{c^{3/2}}
\sum_{m\ll c^2\mu^{1+\varepsilon}/X}\frac{|\lambda_{f}(m)|}{m^{1/2}}T_i^{1/2}.
\end{split}
\end{equation*}
By \eqref{T0}, Cauchy--Schwartz inequality,
the Rankin--Selberg estimate in \eqref{GL2: Rankin Selberg}
and Lemma \ref{LR lemma}, we have
\bea\label{S10}
S_{10}
&\ll&\frac{X^{3/2+\varepsilon}}{C^{3/2}}\sum_{c\sim C_0}c^{-1/2+\varepsilon}c_1^{-1/4}
\sum_{m\ll c^2\mu^{1+\varepsilon}/X}\frac{|\lambda_{f}(m)|}{m^{1/2}}(m,c)^{1/4}\nonumber\\
&\ll&\frac{X^{3/2+\varepsilon}}{C^{3/2}}\sum_{c\sim C_0}c^{-1/2+\varepsilon}c_1^{-1/4}
\left(\sum_{m\ll c^2\mu^{1+\varepsilon}/X}|\lambda_{f}(m)|^2\right)^{1/2}
\left(\sum_{m\ll c^2\mu^{1+\varepsilon}/X}\frac{(m,c)^{1/2}}{m}\right)^{1/2}\nonumber\\
&\ll&\frac{X^{1+\varepsilon}\mu^{1/2+\varepsilon}}{C^{3/2}}
C_0^{5/4+\varepsilon}.
\eea
Similarly, by \eqref{T1}, Cauchy--Schwartz inequality,
the Rankin--Selberg estimate in \eqref{GL2: Rankin Selberg} and
Lemmas \ref{LR lemma}-\ref{x series}, we have
\bea\label{S11}
S_{11}
&\ll&\frac{X^{1+\varepsilon}}{C^{3/2}}\sum_{c\sim C_0}c^{-\frac{1}{2}}c_1^{-\frac{1}{4}+\varepsilon}
\sum_{m\ll c^2\mu^{1+\varepsilon}/X}\frac{|\lambda_{f}(m)|}{m^{1/2}}(m,c)^{\frac{1}{4}}
\left(\sum_{c=c_3c_4}c_4^{-\frac{1}{2}}\sum_{1\leq k\leq \frac{X^{1+\varepsilon}}{c_3}}
\min\left\{X,\left\|2c_3k\alpha\right\|^{-1}\right\}\right)^{1/2}\nonumber\\
&\ll&\frac{X^{1/2+\varepsilon}\mu^{1/2+\varepsilon}}{C^{3/2}}
\left(\sum_{c\sim C_0}c^{1+\varepsilon}c_1^{-1/2+\varepsilon}\right)^{1/2}
\left(\sum_{c_4\ll C_0}c_4^{-1/2}\sum_{c_3\sim C_0/c_4}\sum_{1\leq  k\leq \frac{X^{1+\varepsilon}}{c_3}}
\min\left\{X,\left\|2c_3k\alpha\right\|^{-1}\right\}\right)^{1/2}\nonumber\\
&\ll&\frac{X^{1/2+\varepsilon}\mu^{1/2+\varepsilon}}{C^{3/2}}
C_0^{3/4+\varepsilon}
\left(\sum_{c_4\ll C_0}c_4^{-1/2}\sum_{k'\ll X^{1+\varepsilon}}\tau(k')
\min\left\{X,\left\|2\alpha k'\right\|^{-1}\right\}\right)^{1/2}\nonumber\\
&\ll&\frac{X^{1/2+\varepsilon}\mu^{1/2+\varepsilon}}{C^{3/2}}
C_0^{3/4+\varepsilon}
\left(C_0^{1/2+\varepsilon}
\left(\frac{X^{1+\varepsilon}}{q}+1\right)(X+q\log q)\right)^{1/2}\nonumber\\
&\ll&\frac{X^{1+\varepsilon}\mu^{1/2+\varepsilon}}{C^{3/2}}
C_0^{1+\varepsilon}\left(\frac{X}{q}+\frac{q}{X}\right)^{1/2},
\eea
and the estimate for $S_{12}$ is entirely similar as that for $S_{11}$. Therefore, by \eqref{S10} and \eqref{S11}, we have
\bea\label{S1 upper bound}
S_{1}&\ll&\frac{X^{1+\varepsilon}\mu^{1/2+\varepsilon}}{C^{3/2}}
C_0^{5/4+\varepsilon}+
\frac{X^{1+\varepsilon}\mu^{1/2+\varepsilon}}{C^{3/2}}
C_0^{1+\varepsilon}\left(\frac{X}{q}+\frac{q}{X}\right)^{1/2}.
\eea
Similarly, we get
\bea\label{S2 upper bound}
S_{2}&\ll&\frac{X^{1+\varepsilon}\mu^{1/2+\varepsilon}}{C^{3/2}}C_0^{5/4+\varepsilon}+
\frac{X^{1+\varepsilon}\mu^{1/2+\varepsilon}}{C^{3/2}}
C_0^{1+\varepsilon}\left(\frac{X}{q}+\frac{q}{X}\right)^{1/2},\\
S_{4}&\ll&\frac{X^{3/2+\varepsilon}}{C^2}C_0^{3/4+\varepsilon}
+\frac{X^{3/2+\varepsilon}}{C^2}C_0^{1/2+\varepsilon}
\left(\frac{X}{q}+\frac{q}{X}\right)^{1/2}
\eea
and
\bea\label{S5 upper bound}
S_{3},S_{5}\ll\frac{X^{1/2+\varepsilon}\mu^{1+\varepsilon}}{C}
C_0^{7/4+\varepsilon}+\frac{X^{1/2+\varepsilon}\mu^{1+\varepsilon}}{C}
C_0^{3/2+\varepsilon}\left(\frac{X}{q}+\frac{q}{X}\right)^{1/2}.
\eea
Plugging \eqref{S1 upper bound}--\eqref{S5 upper bound} into \eqref{S(X,alpha,beta,C0,Xi)}
and then inserting the resulting upper bounds into \eqref{C range}, we have
\begin{equation*}
\begin{split}
S(X,\alpha,\beta)
\ll&\frac{X^{9/4+\varepsilon}\mu^{1+\varepsilon}}{C^{11/4}}+
\frac{X^{2+\varepsilon}\mu^{1+\varepsilon}}{C^{5/2}}
\left(\frac{X}{q}+\frac{q}{X}\right)^{1/2}\\
&+X^{1/2+\varepsilon}\mu^{1+\varepsilon}
C^{3/4+\varepsilon}+X^{1/2+\varepsilon}\mu^{1+\varepsilon}
C^{1/2+\varepsilon}\left(\frac{X}{q}+\frac{q}{X}\right)^{1/2}.
\end{split}
\end{equation*}
We take $C=X^{1/2}$ to balance the contribution and obtain
\begin{equation}\label{S(X,alpha,beta)}
\begin{split}
&S(X,\alpha,\beta)
\ll X^{7/8+\varepsilon}\mu^{1+\varepsilon}
+X^{3/4+\varepsilon}\mu^{1+\varepsilon}
\left(\frac{X}{q}+\frac{q}{X}\right)^{1/2}.\\
\end{split}
\end{equation}

We will apply \eqref{S(X,alpha,beta)} for larger $q$.
For smaller $q$, we follow closely Pitt \cite{Pitt}
and combine the results of Godber \cite{Godber}
to give an estimate for the aimed exponential sums
in Theorem \ref{the quadratic exponential sums},
which does not depend on $f$.
\begin{lemma}\label{the eatimate relative to Pitt}
Let $\lambda_f(n)$ be the normalized Fourier coefficients
of a Maass cusp form for $\rm SL_2(\mathbb{Z})$
with Laplacian eigenvalue $\lambda_f(\Delta)=1/4+\mu^2$.
Let $\alpha$ satisfy \eqref{Dirichlet's theorem}. Then for any
$\varepsilon>0$, we have
\bna
\sum_{n \leq X}\lambda_f(n)e\left(\alpha n^2+\beta n\right)\ll
X^{1/2+\varepsilon}q^{1/2}\lambda_f(\Delta)^{1/4+\varepsilon}
+X^{3/2+\varepsilon}Q^{-1/2}\lambda_f(\Delta)^{1/4+\varepsilon},
\ena
where the implied constant depends only on $\varepsilon$.
\end{lemma}
\begin{proof}
By Godber \cite[Theorem 1.2]{Godber},
for any $\gamma \in\mathbb{R}$,
\bea\label{smaller q}
\sum_{n \leq X}\lambda_f(n)e(\gamma n)
\ll X^{1/2+\varepsilon}\lambda_f(\Delta)^{1/4+\varepsilon},
\eea
where the impled constant depends only on $\varepsilon$.
Then Lemma \ref{the eatimate relative to Pitt} holds by
following closely \cite[Section 5]{Pitt} and replacing the estimate
$\sum_{n \leq X}\lambda_f(n)e(\gamma n)
\ll X^{1/2}\log X$
by \eqref{smaller q}.
\end{proof}

\noindent {\bfseries Completion of the proof of Theorem \ref{the quadratic exponential sums}.}
Take $Q=X^{5/4}$.
If the Dirichlet approximant to $\alpha$
has $q\leq X^{3/4}\lambda_f(\Delta)^{1/2}$, we apply
Lemma \ref{the eatimate relative to Pitt},
whereas if $X^{3/4}\lambda_f(\Delta)^{1/2}<q\leq Q$
then we apply $\eqref{S(X,alpha,beta)}$. In either case
we obtain the estimate
$O(X^{7/8+\varepsilon}\lambda_f(\Delta)^{1/2+\varepsilon})$
as in the statement of
Theorem \ref{the quadratic exponential sums}.

\medskip

\section{Proof of Theorem \ref{coefficient sum}}\label{section 2}

By dyadic subdivision, we only need to estimate the sum
\bna
\sum_{n\sim X}\lambda_f(n).
\ena
Let $h$ be a nonnegative smooth function supported in $[1,2]$
such that $h(x)=1$ for $x\in [1+\eta,\, 2-\eta]$ ($\eta>0$)
and $h^{(j)}(x)\ll_j \eta^{-j}$ for any integer $j\geq 0$.
Assume the Ramanujan--Petersson conjecture $\lambda_f(n)\ll n^{\varepsilon}$. Then we have
\bna
\sum_{n\sim X}\lambda_f(n)=\sum_{n\geq 1}\lambda_f(n)h\left(\frac{n}{X}\right)+
O\big(\eta X^{1+\varepsilon}\big).
\ena
By Mellin inversion, one has
$$
h(x)=\frac{1}{2\pi i}\int_{(2)}\widetilde{h}(s)x^{-s}\mathrm{d}s,
$$
where $\widetilde{h}(s)=\int_0^{\infty}h(x)x^{s-1}\mathrm{d}x$.
It follows that
\bea\label{aim-sum-1}
\sum_{n\sim X}\lambda_f(n)=
\frac{1}{2\pi i}\int_{(2)}\widetilde{h}(s)L(s,f)X^{s}\mathrm{d} s+
O\big(\eta X^{1+\varepsilon}\big).
\eea
By shifting the line of integration in \eqref{aim-sum-1} to
Re$(s)=-\varepsilon$, we get
\bea\label{1}
\sum_{n\sim X}\lambda_f(n)=
\frac{X^{-\varepsilon}}{2\pi}\int_{\mathbb{R}}X^{it}\widetilde{h}(-\varepsilon+it)
L(-\varepsilon+it,f)\mathrm{d} t+
O\big(\eta X^{1+\varepsilon}\big).
\eea
Repeated integration by parts shows that for any integer $j\geq 1$,
\bea\label{Tilde-h}
\widetilde{h}(s)=\frac{(-1)^{j}}{s(s+1)\cdots (s+j-1)}\int_{0}^{\infty}
h^{(j)}(x)x^{s+j-1} \mathrm{d} x \ll \eta(\eta|s|)^{-j}.
\eea
Thus for $T\geq 1$ and $j>2$,
\bna
&&\int_{|t|\geq T}X^{it}\widetilde{h}(-\varepsilon+it)
L(-\varepsilon+it,f)\mathrm{d} t\\
&\ll&\int_{|t|\geq T}\eta(\eta|t|)^{-j}
\lambda_f(\Delta)^{1/2+\varepsilon}|t|^{1+\varepsilon}\mathrm{d} t\\
&\ll&\lambda_f(\Delta)^{1/2+\varepsilon}\eta T^2 (\eta T)^{-j},
\ena
where we have used the convexity bound (see \cite[(5.20)]{IK})
\bna
L(\sigma+it,f)
\ll_{\sigma,\varepsilon}
(\left(1+|t+\mu|\right)\left(1+|t-\mu|\right))^{(1-\sigma)/2+\varepsilon}.
\ena
Consequently, by taking $j$ sufficiently large, the contribution from
$T\geq (\lambda_f(\Delta)X)^{\varepsilon}\eta^{-1}$ is negligible.
Moreover, by taking $j=1$ in \eqref{Tilde-h},
for $T_1=X^{1/3}\lambda_f(\Delta)^{-\theta}$,
where $\theta$ is a parameter to be optimized later,
\bna
&&\int_{|t|\leq T_1}X^{it}\widetilde{h}(-\varepsilon+it)
L(-\varepsilon+it,f)\mathrm{d} t\\
&\ll&\int_{|t|\leq T_1}(1+|t|)^{-1}
\lambda_f(\Delta)^{1/2+\varepsilon}(1+|t|)^{1+\varepsilon}\mathrm{d} t\\
&\ll&\lambda_f(\Delta)^{1/2-\theta+\varepsilon} X^{1/3}.
\ena
Therefore, by inserting a dyadic smooth partition of unit to the $t$-integral in \eqref{1}, one has
\bea\label{dyadic}
\sum_{n\sim X}\lambda_f(n)=
\frac{X^{-\varepsilon}}{2\pi}\sum_{T_1\ll T\ll (\lambda_f(\Delta)X)^{\varepsilon}\eta^{-1}
\atop T\, \text{dyadic}}\mathcal{I}(T)+
O\left(\eta X^{1+\varepsilon}+
\lambda_f(\Delta)^{1/2-\theta+\varepsilon} X^{1/3}\right),
\eea
where
\bea\label{I(T)}
\mathcal{I}(T):=\int_{\mathbb{R}}X^{it}\widetilde{h}(-\varepsilon+it)
L(-\varepsilon+it,f)\omega\left(\frac{|t|}{T}\right)\mathrm{d} t
\eea
with $\omega(x)\in C_c^{\infty}(1,2)$ satisfying $\omega^{(j)}(x)\ll_j 1$ for any integer $j\geq 0$.

For $T>|\mu|^{1+\varepsilon}$, we apply the functional equation \eqref{functional equation}
for $L(-\varepsilon+it,f)$, introduce
the series \eqref{Dirichlet series} and then integrate termwise to obtain
\bna
\mathcal{I}(T)
&\ll& \sup_{x\in[1,\, 2]} \bigg|\int_{\mathbb{R}}\frac{(Xx)^{it}}{\varepsilon-it}\gamma(1+\varepsilon-it)
L(1+\varepsilon-it,f)\omega\left(\frac{|t|}{T}\right)\mathrm{d} t \bigg|\nonumber\\
&\ll&\sup_{x\in[1,\, 2]} \bigg|\sum_{n=1}^{\infty}\frac{\lambda_f(n)}{n^{1+\varepsilon}}
\int_{\mathbb{R}}\frac{(Xxn)^{it}}{\varepsilon-it}\gamma(1+\varepsilon-it)
\omega\left(\frac{|t|}{T}\right)\mathrm{d} t \bigg|,
\ena
where $\gamma(s)$ is defined in \eqref{Gamma}.
By Stirling's formula in \eqref{Stirling approximation},
for $|t|\asymp T>|\mu|^{1+\varepsilon}$,
\bea\label{Gamma-asymptotic}
\gamma(\sigma+it)&=&\pi^{1-2\sigma-2it}\prod_{\pm}
\frac{\Gamma\left(\frac{\sigma+\eta+i(t\pm \mu)}{2}\right)}{
\Gamma\left(\frac{1-\sigma+\eta-i(t\pm\mu)}{2}\right)}\nonumber\\
&=&\pi^{1-2\sigma-2it}\prod_{\pm}\bigg(
(t\pm \mu)^{\sigma-\frac{1}{2}}
g_{\sigma}(t\pm \mu)\bigg)
\exp\bigg(i\sum_{\pm}(t\pm \mu)\log\frac{|t\pm \mu|}{2e}\bigg),
\eea
where $g_{\sigma}(t)$ satisfies $t^{\ell}g_{\sigma}^{(\ell)}(t)\ll_{\ell,\sigma} 1. $
It follows that
\bea\label{upper bound}
\gamma(\sigma+it)
\ll \big(t^2+\mu^2\big)^{\sigma-\frac{1}{2}}\ll t^{2\sigma-1}.
\eea
Now we use \eqref{upper bound} together with $(\varepsilon-it)^{-1}=it^{-1}+O(t^{-2})$ to derive
\bea\label{I(T)-2}
\mathcal{I}(T)
\ll\sup_{x\in[1,\, 2]} \bigg|\sum_{n=1}^{\infty}\frac{\lambda_f(n)}{n^{1+\varepsilon}}
\mathfrak{J}(Xxn,T) \bigg|
+T^{\varepsilon},
\eea
where
\bna
\mathfrak{J}(P,T)=
\int_{\mathbb{R}}P^{it}\gamma(1+\varepsilon-it)
\omega\left(\frac{|t|}{T}\right)\frac{\mathrm{d}t}{t}.
\ena
Making a change of variable $t\rightarrow Tt$, one has
\bna
\mathfrak{J}(P,T)=
\int_{\mathbb{R}}P^{iTt}\gamma(1+\varepsilon-iTt)
\omega\left(|t|\right)\frac{\mathrm{d}t}{t}.
\ena
By \eqref{Gamma-asymptotic} we have
\bna
\mathfrak{J}(P,T)=
\int_{\mathbb{R}}
u(t)
e(\varsigma(t))\mathrm{d} t,
\ena
where $u(t)=\pi^{-1-2\varepsilon}
t^{-1}(t^2T^2-\mu^2)^{1/2+\varepsilon}\omega\left(|t|\right)\prod_{\pm}g_{1+\varepsilon}(-tT\pm\mu)$ and
\bna
\varsigma(t)=\frac{tT}{2\pi}\log (\pi^2P)-\frac{1}{2\pi}\sum_{\pm}(tT\pm \mu)\log\frac{|tT\pm \mu|}{2e}.
\ena
We can apply the stationary phase
analysis to $\mathfrak{J}(P,T)$. Without loss of generality, we consider the case
$t>0$, since the case $t<0$ can be treated similarly.
Note that
\bna
\varsigma'(t)=-\frac{T}{2\pi}\log \frac{t^2T^2-\mu^2}{4\pi^2P}
\ena
and
\bna
\varsigma''(t)=-\frac{T}{2\pi}\frac{2tT^2}{t^2T^2-\mu^2}\asymp T.
\ena
We also have $u^{(j_1)}(t)\ll_{j_1} T^{1+\varepsilon}$ and
$\varsigma^{(j_2)}(t)\ll_{j_2} T$
for any integer $j_1\geq 0$ and $j_2\geq 2$.
Then integration by parts shows that the integral is negligibly small unless
$T\asymp \sqrt{P}$.
The stationary point which is the solution to the equation $\varsigma'(t)=0$ is
$t_0=T^{-1}(\mu^2+4\pi^2P)^{1/2}$.
By Lemma \ref{lemma:exponentialintegral}, we have
\bea\label{stationary phase}
\mathfrak{J}(P,T)=\int_{\mathbb{R}}u(t)e(\varsigma(t))\mathrm{d} t
=T^{1/2+2\varepsilon}F_{\natural}(t_0)e(\varsigma(t_0))
+O_A\Big(T^{-A}\Big),
\eea
where $F_{\natural}(x)$
is an $1$-inert function\,(depending on $A$) supported on $x\asymp 1$.
Plugging \eqref{stationary phase} into \eqref{I(T)-2} and estimating the resulting sum over $n$ trivially,
we obtain
\bea\label{I(T)-3}
\mathcal{I}(T)
\ll T^{1/2+\varepsilon}.
\eea

If $T_1>|\mu|^{1+\varepsilon}$, then $T\gg T_1>|\mu|^{1+\varepsilon}$.
If $T_1\leq|\mu|^{1+\varepsilon}$,
i.e., $X^{1/3}\lambda_f(\Delta)^{-\theta}
\leq|\mu|^{1+\varepsilon}$, then
\bea\label{X upper bound}
X\leq\lambda_f(\Delta)^{3(1/2+\theta)}.
\eea
In this case, for $T\leq|\mu|^{1+\varepsilon}$,
we move the line of integration in \eqref{I(T)}
to $\Re(s)=1/2$ and apply \eqref{Tilde-h}
with $j=1$ to get
\bna
\mathcal{I}(T)&=&\int_{\mathbb{R}}
X^{\frac{1}{2}+it+\varepsilon}\widetilde{h}\left(\frac{1}{2}+it\right)
L\left(\frac{1}{2}+it,f\right)
\omega\left(\frac{|\frac{1}{2}+it+\varepsilon|}{T}\right)\mathrm{d} t\\
&=&-X^{\frac{1}{2}+\varepsilon}
\int_1^2
h'(x)x^{\frac{1}{2}}\int_{\mathbb{R}}
\frac{(xX)^{it}}{1/2+it}
L\left(\frac{1}{2}+it,f\right)
\omega\left(\frac{|\frac{1}{2}+it+\varepsilon|}{T}\right)
\mathrm{d}t\mathrm{d}x.
\ena
Applying the subconvexity bound due to Jutila and
Motohashi \cite{JM},
$$L\left(\frac{1}{2}+it,f\right)\ll(|t|+\mu)^{1/3+\varepsilon},$$
we obtain
\bna
\mathcal{I}(T)\ll X^{1/2+\varepsilon}
\int_{|t|\asymp T}\frac{1}{t}
(|t|+\mu)^{1/3+\varepsilon}\mathrm{d}t
\ll X^{1/2+\varepsilon}\lambda_f(\Delta)^{1/6+\varepsilon},
\ena
 which by \eqref{X upper bound} is bounded
 by $X^{1/3}\lambda_f(\Delta)^{\theta/2+5/12+\varepsilon}$.
This estimate combined with \eqref{I(T)-3}
when plugged into \eqref{dyadic} implies
\bna
\sum_{n\sim X}\lambda_f(n)\ll (\lambda_f(\Delta)X)^{\varepsilon}\eta^{-1/2}
+\eta X^{1+\varepsilon}
+\lambda_f(\Delta)^{1/2-\theta+\varepsilon} X^{1/3}
+\lambda_f(\Delta)^{\theta/2+5/12+\varepsilon}X^{1/3}.
\ena
Take $\eta=X^{-2/3}$ and $\theta=\frac{1}{18}$.
We conclude that
\bna
\sum_{n\sim X}\lambda_f(n)\ll
X^{1/3+\varepsilon}\lambda_f(\Delta)^{4/9+\varepsilon}.
\ena

\bigskip

\bibliographystyle{amsplain}

\end{document}